\documentclass[11pt, a4paper]{article}

\usepackage{amsmath, amsthm, amssymb, url}
\usepackage[margin=1in]{geometry}
\usepackage{multirow}
\usepackage{tikz, tkz-graph, mathrsfs, enumerate}

\usetikzlibrary{arrows}
\newcommand{\vertex}[3]{\node [vertex] (#1) at (#2, #3 * 1.7) {};}

\newcommand{\Tran}{\mathrm{Tran}}
\newcommand{\Sym}{\mathrm{Sym}}

\newcommand{\id}{\mathrm{id}}
\newcommand{\Char}{\mathrm{char}}
\newcommand{\CA}{\mathrm{CA}}
\newcommand{\ICA}{\mathrm{ICA}}
\newcommand{\LCA}{\mathrm{LCA}}

\newcommand{\End}{\mathrm{End}}

\newcommand{\rev}{\mathrm{rev}}

\theoremstyle{plain}

\newtheorem{corollary}{Corollary}
\newtheorem{lemma}{Lemma}
\newtheorem{proposition}{Proposition}
\newtheorem{theorem}{Theorem}
\newtheorem*{claim*}{Claim}

\theoremstyle{definition}
\newtheorem{definition}{Definition}

\newtheorem{example}{Example}
\newtheorem{remark}{Remark}
\newtheorem{question}{Question}

\begin{document}

\title{Elementary, Finite and Linear vN-Regular Cellular Automata}
\author{Alonso Castillo-Ramirez and Maximilien Gadouleau}

\newcommand{\Addresses}{{
  \bigskip
  \footnotesize

A. Castillo-Ramirez (Corresponding author), \textsc{Universidad de Guadalajara, CUCEI, Departamento de Matem\'aticas, Guadalajara, M\'exico.} \par \nopagebreak
Email: \texttt{alonso.castillor@academicos.udg.mx}

\medskip

M. Gadouleau, \textsc{Department of Computer Science, Durham University, South Road, Durham DH1 3LE, U.K.} \par \nopagebreak
 Email: \texttt{m.r.gadouleau@durham.ac.uk}
}}

\maketitle

\begin{abstract}
Let $G$ be a group and $A$ a set. A cellular automaton (CA) $\tau$ over $A^G$ is von Neumann regular (vN-regular) if there exists a CA $\sigma$ over $A^G$ such that $\tau \sigma\tau = \tau$, and in such case, $\sigma$ is called a weak generalised inverse of $\tau$. In this paper, we investigate vN-regularity of various kinds of CA. First, we establish that, over any nontrivial configuration space, there always exist CA that are not vN-regular. Then, we obtain a partial classification of elementary vN-regular CA over $\{ 0,1\}^{\mathbb{Z}}$; in particular, we show that rules like 128 and 254 are vN-regular (and actually generalised inverses of each other), while others, like the well-known rules $90$ and $110$, are not vN-regular. Next, when $A$ and $G$ are both finite, we obtain a full characterisation of vN-regular CA over $A^G$. Finally, we study vN-regular linear CA when $A= V$ is a vector space over a field $\mathbb{F}$; we show that every vN-regular linear CA is invertible when $V= \mathbb{F}$ and $G$ is torsion-free elementary amenable (e.g. when $G=\mathbb{Z}^d, \ d \in \mathbb{N}$), and that every linear CA is vN-regular when $V$ is finite-dimensional and $G$ is locally finite with $\Char(\mathbb{F}) \nmid o(g)$ for all $g \in G$.    
 \\
\\
\textbf{Keywords}: Cellular automata, elementary cellular automata, finite cellular automata, linear cellular automata, monoids, von Neumann regularity, generalised inverse.
\end{abstract}


\section{Introduction} \label{sec:introduction}

In this paper we follow the general setting for cellular automata (CA) presented in \cite{CSC10}. For any group $G$ and any set $A$, the \emph{configuration space} $A^G$ is the set of all functions from $G$ to $A$. A \emph{cellular automaton} over $A^G$ is a transformation of $A^G$ defined via a finite memory set and a local function (see Definition \ref{def:ca} for the precise details). Most of the classical literature on CA focus on the case when $G=\mathbb{Z}^d$, for $d\geq1$, and $A$ is a finite set (see \cite{Ka05}), but important results have been obtained for larger classes of groups (e.g., see \cite{CSC10} and references therein).

Recall that a \emph{semigroup} is a set equipped with an associative binary operation, and that a \emph{monoid} is a semigroup with an identity element. If $M$ is a group, or a monoid, write $K \leq M$ if $K$ is a subgroup, or a submonoid, of $M$, respectively. 

Let $\CA(G;A)$ be the set of all CA over $A^G$. It turns out that, equipped with the composition of functions, $\CA(G;A)$ is a monoid. In this paper we apply functions on the right; hence, for $\tau,\sigma \in \CA(G;A)$, the composition $\tau \circ \sigma$, denoted simply by $\tau \sigma$, means applying first $\tau$ and then $\sigma$. 

A cellular automaton $\tau \in \CA(G;A)$ is \emph{invertible}, or \emph{reversible}, or \emph{a unit}, if there exists $\sigma \in \CA(G;A)$ such that $\tau \sigma = \sigma \tau = \id$. In such case, $\sigma$ is called \emph{the inverse} of $\tau$ and denoted by $\sigma = \tau^{-1}$. When $A$ is finite, it may be shown that $\tau \in \CA(G;A)$ is invertible if and only if it is a bijective function (see \cite[Theorem 1.10.2]{CSC10}). Denote the set of all invertible CA over $A^G$ by $\ICA(G;A)$; this is in fact a group, called the \emph{group of units} of $\CA(G;A)$. 

We shall consider the notion of regularity that, as cellular automata, was introduced by John von Neumann, and has been widely studied in both semigroup and ring theory. A cellular automaton $\tau \in \CA(G;A)$ is \emph{von Neumann regular} (vN-regular) if there exists $\sigma \in \CA(G;A)$ such that $\tau \sigma \tau = \tau$; in this case, $\sigma$ is called a \emph{weak generalised inverse} of $\tau$. Equivalently, $\tau \in \CA(G;A)$ is vN-regular if and only if there exists $\sigma \in \CA(G;A)$ mapping every configuration in the image of $\tau$ to one of its preimages under $\tau$ (see Lemma \ref{le-vN-regular}). Clearly, the notion of vN-regularity generalises reversibility.  

In general, for any semigroup $S$ and $a, b \in S$, we say that $b$ is \emph{a weak generalised inverse} of $a$ if 
\[ aba=a. \]
We say that $b$ is \emph{a generalised inverse} (often just called \emph{an inverse}) of $a$ if
\[ aba = a \text{ and } bab= b. \]
An element $a \in S$ may have none, one, or more (weak) generalised inverses. It is clear that any generalised inverse of $a$ is also a weak generalised inverse; not so obvious is that, given the set $W(a)$ of weak generalised inverses of $a$ we may obtain the set $V(a)$ of generalised inverses of $a$ as follows (see \cite[Exercise 1.9.7]{CP61}):
\[ V(a) = \{ b a b^\prime : b, b^\prime \in W(a)  \}. \]
Thus, an element $a \in S$ is \emph{vN-regular} if it has at least one generalised inverse (which is equivalent to having at least one weak generalised inverse). The semigroup $S$ itself is called \emph{vN-regular} if all of its elements are vN-regular. Many of the well-known types of semigroups are vN-regular, such as idempotent semigroups (or \emph{bands}), full transformation semigroups, and Rees matrix semigroups. Among various advantages, vN-regular semigroups have a particularly manageable structure which may be studied using Green's relations. For further basic results on vN-regular semigroups see \cite[Section 1.9]{CP61}. 

Another generalisation of reversible CA has appeared in the literature before \cite{ZZ12,ZZ15} using the concept of \emph{Drazin inverse} \cite{D58}. However, as Drazin invertible elements are a special kind of vN-regular elements, our approach turns out to be more general and natural.  

In the following sections we study the vN-regular elements in monoids of CA. First, in Section \ref{vN-regular} we present some basic results and examples, and we establish that, except for the trivial cases $\vert G \vert = 1$ and $\vert A \vert = 1$, the monoid $\CA(G;A)$ is not vN-regular. 

In Section \ref{sec-elementary}, we obtain a partial classification of the vN-regular elementary CA in $\CA(\mathbb{Z};\{ 0,1 \})$. We divide the $256$ elementary CA into $48$ equivalence classes that preserve vN-regularity: this is an extension of the usual division of elementary CA into $88$ equivalence classes that preserve dynamical properties. Among various results, we show that rules like 128 and 254 are vN-regular (and actually generalised inverses of each other), while others, like the well-known rules $90$ and $110$, are not vN-regular. Our classification is only partial as the vN-regularity of $11$ classes could not be determined (see Table \ref{elementary}). 

In Section \ref{finite}, we study the vN-regular elements of $\CA(G;A)$ when $G$ and $A$ are both finite; in particular, we characterise them and describe a vN-regular submonoid. 

Finally, in Section \ref{linear}, we study the vN-regular elements of the monoid $\LCA(G; V)$ of linear CA, when $V$ is a vector space over a field $\mathbb{F}$. Specifically, using results on group rings, we show that, when $G$ is torsion-free elementary amenable (e.g., $G=\mathbb{Z}^d$), $\tau \in \LCA(G; \mathbb{F})$ is vN-regular if and only if it is invertible, and that, for finite-dimensional $V$, $\LCA(G; V)$ itself is vN-regular if and only if $G$ is locally finite and $\Char(\mathbb{F}) \nmid \vert \langle g \rangle \vert$, for all $g \in G$. Finally, for the particular case when $G \cong \mathbb{Z}_n$ is a cyclic group, $V := \mathbb{F}$ is a finite field, and $\Char(\mathbb{F}) \mid n$, we count the total number of vN-regular elements in $\LCA(\mathbb{Z}_n ;\mathbb{F})$.  

The present paper is an extended version of \cite{CRG17a}. Besides improving the general exposition, Theorem \ref{test-non-regular}, Section \ref{sec-elementary} and Theorem \ref{structure-monoid} are completely new, and Theorem \ref{torsion-free} is corrected (as the proof of Theorem 5 in \cite{CRG17a} is flawed).


\section{vN-regular cellular automata} \label{vN-regular}

For any set $X$, let $\Tran(X)$ and $\Sym(X)$ be the sets of all functions and all bijective functions of the form $\tau : X \to X$, respectively. Equipped with the composition of functions, $\Tran(X)$ is known as the \emph{full transformation monoid} on $X$, and $\Sym(X)$ is the \emph{symmetric group} on $X$. When $X$ is a finite set of size $\alpha$, we simply write $\Tran_\alpha$ and $\Sym_\alpha$, in each case.

We shall review the broad definition of CA that appears in \cite[Sec.~1.4]{CSC10}. Let $G$ be a group and $A$ a set. Denote by $A^G$ the \emph{configuration space}, i.e. the set of all functions of the form $x:G \to A$. The group $G$ acts on the configuration space $A^G$ as follows: for each $g \in G$ and $x \in A^G$, the configuration $x \cdot g \in A^G$ is defined by 
\[ (h)x \cdot g := (hg^{-1})x, \quad \forall h \in G. \] 

The following definitions shall be usefuls in the rest of this paper.

\begin{definition} Let $G$ be a group and $A$ a set.
\begin{enumerate}
\item For any $x \in A^G$, the \emph{$G$-orbit} of $x$ in $A^G$ is $xG := \{ x \cdot g : g \in G  \}$.
\item For any $x \in A^G$, the \emph{stabiliser} of $x$ in $G$ is $G_x := \{ g \in G : x \cdot g = x \}$.
 \item A \emph{subshift} of $A^G$ is a subset $X \subseteq A^G$ that is \emph{$G$-invariant}, i.e. for all $x \in X$, $g \in G$, we have $x \cdot g \in X$, and closed in the prodiscrete topology of $A^G$. 
\end{enumerate}
\end{definition}

\begin{definition} \label{def:ca}
Let $G$ be a group and $A$ a set. A \emph{cellular automaton} over $A^G$ is a transformation $\tau : A^G \to A^G$ satisfying the following: there is a finite subset $S \subseteq G$, called a \emph{memory set} of $\tau$, and a \emph{local function} $\mu : A^S \to A$ such that, for all $x \in A^G, g \in G$,
\[ (g)(x)\tau = (( x \cdot g^{-1}  )\vert_{S}) \mu,   \]
where $( x \cdot g^{-1} )\vert_{S}$ is the restriction to $S$ of $x \cdot g^{-1} \in A^G$ .
\end{definition} 

We emphasise that we apply functions on the right, while \cite{CSC10} applies functions on the left.   

A transformation $\tau : A^G \to A^G$ is \emph{$G$-equivariant} if, for all $x \in A^G$, $g \in G$,
\[ (x \cdot g) \tau = ( (x) \tau ) \cdot g .\] 
Any cellular automaton is $G$-equivariant, but the converse is not true in general. A generalisation of Curtis-Hedlund Theorem (see \cite[Theorem 1.8.1]{CSC10}) establishes that, when $A$ is finite, $\tau : A^G \to A^G$ is a CA if and only if $\tau$ is $G$-equivariant and continuous in the prodiscrete topology of $A^G$; in particular, when $G$ and $A$ are both finite, $G$-equivariance completely characterises CA over $A^G$. 

A configuration $x \in A^G$ is called \emph{constant} if $(g)x = k$, for a fixed $k \in A$, for all $g \in G$. In such case, we denote $x$ by $\mathbf{k} \in A^G$. 

\begin{remark}\label{constant}
It follows by $G$-equivariance that any $\tau \in \CA(G;A)$ maps constant configurations to constant configurations.  
\end{remark}

Recall from Section \ref{sec:introduction} that $\tau \in \CA(G;A)$ is \emph{invertible} if there exists $\sigma \in \CA(G;A)$ such that $\tau \sigma = \sigma \tau = \id$, and that $\tau \in \CA(G;A)$ is \emph{vN-regular} if there exists $\sigma \in \CA(G; A)$ such that $\tau \sigma \tau = \tau$. Clearly, every invertible CA is vN-regular. We now present some examples of CA that are vN-regular but not invertible.

\begin{example}
If $\tau \in \CA(G;A)$ is idempotent (i.e. $\tau^2 = \tau$), then it is vN-regular as $\tau \tau \tau = \tau$. 
\end{example}

\begin{example}
Let $G$ be any nontrivial group and $A$ any set with at least two elements. Let $\sigma \in \CA(G;A)$ be a CA with memory set $\{s \} \subseteq G$ and local function $\mu : A \to A$. As $\Tran(A)$ is vN-regular, there exists $\mu^\prime : A \to A$ such that $\mu \mu^\prime \mu = \mu$. If $\sigma^\prime : A^G \to A^G$ is the CA with memory set $\{ s^{-1} \}$ and local function $\mu^\prime$, then $\sigma \sigma^\prime \sigma = \sigma$. Hence $\sigma$ is vN-regular. In particular, when $\mu$ is not bijective, this gives an example of a non-invertible CA $\sigma$ that is vN-regular. 
\end{example}

\begin{example}\label{ex:rules128-254}
Suppose that $A = \{0,1, \dots, q-1\}$, with $q \geq 2$. Consider $\tau_1, \tau_2 \in \CA(\mathbb{Z};A)$ with memory set $S:=\{-1,0,1 \}$ and local functions
\[ (x)\mu_1 = \min \{  (-1)x, (0)x, (1)x \} \text{ and } \ (x)\mu_2 = \max \{ (-1)x, (0)x ,  (1)x \}, \]
respectively, for all $x \in A^S$. In particular, when $q=2$, $\tau_1$ and $\tau_2$ are the elementary CA known as Rules 128 and 254, respectively.
Clearly, $\tau_1$ and $\tau_2$ are not invertible, but we show that they are generalised inverses of each other, i.e.  $\tau_1 \tau_2 \tau_1 = \tau_1$ and $\tau_2 \tau_1 \tau_2 = \tau_2$, so they are both vN-regular. We prove only the first of the previous identities, as the second one is symmetrical. 

Consider 
\[ x \in A^\mathbb{Z}, \ \  y := (x)\tau_1, \ \  z:= (y)\tau_2, \ \text{ and } a := (z)\tau_1. \]
We want to show that $y = a$. By equivariance, it is enough to show that $(0)y=(0)a$. For $\epsilon \in \{ -1,0,1 \}$, we have
\[ (\epsilon )y = \min\{ (\epsilon - 1) x, (\epsilon)x, (\epsilon + 1)x \} \leq (0)x.\]
Hence,
\[(0)z = \max\{ (-1)y, (0)y , (1)y \} \leq (0)x. \]
Similarly $(-1)z \leq (-1)x$ and $(1)z \leq (1)x$, so
\[ (0)a = \min\{ (-1)z, (0)z , (1)z  \} \le (0)y = \min \{ (-1)x, (0)x, (1)x \}. \]
Conversely, we have $(-1)z, (0)z, (1)z \ge (0)y$, so $(0)a \ge (0)y$.
\end{example}

The following lemma gives an equivalent definition of vN-regular CA. Note that this result still holds if we replace $\CA(G;A)$ with any monoid of transformations. 

\begin{lemma}\label{le-vN-regular}
Let $G$ be a group and $A$ a set. Then, $\tau \in \CA(G;A)$ is vN-regular if and only if there exists $\sigma \in \CA(G;A)$ such that for every $y \in (A^G) \tau$ there is $\hat{y} \in A^G$ with $(\hat{y})\tau = y$ and $(y)\sigma = \hat{y}$.
\end{lemma}
\begin{proof}
If $\tau \in \CA(G;A)$ is vN-regular, there exists $\sigma \in \CA(G;A)$ such that $\tau \sigma \tau = \tau$. Let $x \in A^G$ be such that $(x)\tau = y$ (which exists because $y \in (A^G) \tau$) and define $\hat{y}:= (y)\sigma$. Now,
\[ (\hat{y})\tau = (y)\sigma\tau = (x)\tau\sigma\tau = (x)\tau = y. \]
Conversely, assume there exists $\sigma \in \CA(G;A)$ satisfying the statement of the lemma. Then, for any $x \in A^G$ with $y:=(x)\tau$ we have
\[ (x) \tau \sigma \tau = (y) \sigma \tau = (\hat{y})\tau = y = (x)\tau. \]
Therefore, $\tau$ is vN-regular. 
\end{proof}

The following is a powerful tool to show that a CA is not vN-regular.

\begin{theorem}\label{test-non-regular}
Let $G$ be a group, $A$ a set, and $\tau \in \CA(G;A)$. Suppose there exists $x \in A^{G}$ such that
\[ x \in  (A^{G})\tau, \text{  but  } x \neq (y)\tau \text{ for all } y\in A^G \text{ such that } G_x = G_y. \]
Then, $\tau$ is not vN-regular. 
\end{theorem}
\begin{proof}
First, note that for any $\sigma \in \CA(G;A)$ and $z \in A^G$ we have $G_z \leq G_{(z)\sigma}$. Indeed, for any $g \in G_z$ we have $(z)\sigma \cdot g = (z \cdot g) \sigma = (z) \sigma$, so $g \in G_{(z)\sigma}$.

For a contradiction, suppose that $\tau$ is vN-regular. By Lemma \ref{le-vN-regular}, there exists $\sigma \in \CA(G;A)$ mapping $x$ to one of its preimages under $\tau$: say $z \in A^G$ satisfies $(z)\tau = x$ and $(x)\sigma = z$. By the above paragraph, $G_z \leq G_{(z)\tau} = G_{x}$ and $G_{x}\leq G_{(x)\sigma} = G_{z} $, so $G_z  = G_x$. This contradicts the hypothesis. Thus, $\tau$ is not vN-regular.
\end{proof}

\begin{corollary}\label{cor-vN-regular}
Let $G$ be a group and $A$ a set. Let $\tau \in \CA(G;A)$, and suppose there is a constant configuration $\mathbf{k} \in (A^G)\tau$ such that $\mathbf{k} \neq (\mathbf{s})\tau$ for all constant configurations $\mathbf{s} \in A^G$. Then $\tau$ is not vN-regular.
\end{corollary}
\begin{proof}
This follows by Theorem \ref{test-non-regular} and the fact that $x \in A^G$ is constant if and only if $G_x = G$.  
\end{proof}

In the following examples we see how Corollary \ref{cor-vN-regular} may be used to show that some well-known CA are not vN-regular.

\begin{example}\label{ex:rule110}
Let $\phi \in \CA(\mathbb{Z}; \{0,1 \})$ be the Rule 110 elementary CA, and consider the constant configuration $\mathbf{1}$. Define $x := \dots 10101010 \dots \in \{0,1 \}^{\mathbb{Z}}$, and note that $(x)\phi = \mathbf{1}$. Since $(\mathbf{1})\phi = \mathbf{0}$ and $(\mathbf{0})\phi= \mathbf{0}$, Corollary \ref{cor-vN-regular} implies that $\phi$ is not vN-regular. 
\end{example}

\begin{example}
Let $\tau \in \CA(\mathbb{Z}^2; \{0,1 \})$ be Conway's Game of Life, and consider the constant configuration $\mathbf{1}$ (all cells alive). By \cite[Exercise 1.7.]{CSC10}, $\mathbf{1}$ is in the image of $\tau$; since $(\mathbf{1})\tau = \mathbf{0}$ (all cells die from overpopulation) and $(\mathbf{0})\tau = \mathbf{0}$, Corollary \ref{cor-vN-regular} implies that $\tau$ is not vN-regular.
\end{example}

The following theorem applies to CA over arbitrary groups and sets, and it shows that, except for the trivial cases, $\CA(G;A)$ always contains elements that are not vN-regular.

\begin{theorem} \label{th:vN-regular}
Let $G$ be a group and $A$ a set. The monoid $\CA(G;A)$ is vN-regular if and only if $\vert G \vert = 1$ or $\vert A \vert = 1$.
\end{theorem}
\begin{proof}
If $\vert G \vert = 1$ or $\vert A \vert = 1$, then $\CA(G;A) = \Tran(A)$ or $\CA(G;A)$ is the trivial monoid with one element, respectively. In both cases, $\CA(G;A)$ is vN-regular (see \cite[Exercise 1.9.1]{CP61}).

Assume that $\vert G \vert \geq 2$ and $\vert A \vert \geq 2$. Suppose that $\{ 0,1\} \subseteq A$. Let $S := \{e,g,g^{-1}\} \subseteq G$, where $e$ is the identity of $G$ and $e \neq g \in G$ (we do not require $g \neq g^{-1}$). For $i =1,2$, let $\tau_i \in \CA(G;A)$ be the cellular automaton defined by the local function $\mu_i : A^S \to A$ defined by
\begin{align*}
(x)\mu_1   &:=  \begin{cases}
(e)x & \text{if } (e)x = (g)x = (g^{-1})x , \\
0 & \text{otherwise};
\end{cases}  \\
(x) \mu_2 &:= \begin{cases}
1 & \text{if } (e)x = (g)x = (g^{-1})x= 0 , \\
(e)x & \text{otherwise},
\end{cases} 
\end{align*}
for any $x \in A^S$. We shall show that $\tau := \tau_2 \tau_1 \in \CA(G;A)$ is not vN-regular.

Consider the constant configurations $\mathbf{0}, \mathbf{1} \in A^G$. Let $z \in A^G$ be defined by
\[ (h)z := \begin{cases}
m \mod(2) & \text{if } h=g^m, m \in \mathbb{N} \text{ minimal}, \\
0 & \text{ otherwise}. 
\end{cases} \]   

\begin{figure}[h]
\centering
\begin{tikzpicture}[vertex/.style={circle, draw, fill=none, inner sep=0.2cm}]
    \vertex{1}{1}{2}    \node at (1,3.4) {$z$};  
    \vertex{2}{3}{1}    \node at (3,1.7) {$\mathbf{k}$};   
   \vertex{3}{1}{1}    \node at (1,1.7) {$\mathbf{0}$};  
   \vertex{4}{1}{0}   \node at (1,0) {$\mathbf{1}$};   

	\path[->,every loop/.style={min distance=6mm,in=65,out=120,looseness=5}] (1) edge [loop above] node[pos=.5,above] {\small{$\tau_2$}} () ;
	\path[->,every loop/.style={min distance=6mm,in=65,out=120,looseness=5}] (2) edge [loop above] node[pos=.4,above] {\small{\ \ $\tau_1,\tau_2$}} () ;
	\path[->,every loop/.style={min distance=6mm,in=210,out=155,looseness=5}] (3) edge [loop above] node[pos=.5,left] {\small{$\tau_1$}} () ;
	\path[->,every loop/.style={min distance=6mm,in=-30,out=25,looseness=5}] (4) edge [loop above] node[pos=.5,right] {\small{$\tau_1, \tau_2$}} () ;

	\draw[->] (1) -- (3) node[pos=.5,right] {\small{$\tau_1$}}; 
	\draw[->] (3) -- (4) node[pos=.5,right] {\small{$\tau_2$}};

   \end{tikzpicture} 
\caption{Images of $\tau_1$ and $\tau_2$.}
\label{Fig1}
   \end{figure}
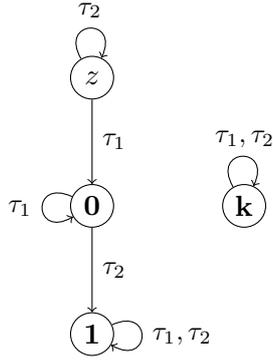

Figure \ref{Fig1} illustrates the images $z$, $\textbf{0}$, $\textbf{1}$, and $\mathbf{k} \neq \textbf{0}, \textbf{1}$ (in case it exists) under $\tau_1$ and $\tau_2$. Clearly,
\[ (\mathbf{0})\tau = ( \mathbf{0})\tau_2 \tau_1 = (\mathbf{1})\tau_1 = \mathbf{1}. \]
In fact, 
 \[ (\mathbf{k})\tau = \begin{cases}
 \mathbf{1} & \text{if } \mathbf{k}=\mathbf{0},  \\
 \mathbf{k} & \text{otherwise}. 
  \end{cases} \] 
Furthermore,  
\[ (z)\tau = (z) \tau_2 \tau_1 = (z)\tau_1 = \textbf{0}. \]
Hence, $\mathbf{0}$ is a constant configuration in the image of $\tau$ but with no preimage among the constant configurations. By Corollary \ref{cor-vN-regular}, $\tau$ is not vN-regular. 
\end{proof}

Now that we know that $\CA(G;A)$ always contains some elements that are vN-regular and some that are not vN-regular (when $\vert G \vert \geq 2$ and $\vert A \vert \geq 2$), an interesting problem is to find a criterion that describes all vN-regular CA. In the following sections, we study this problem in three particular cases: the elementary, the finite and the linear cases.


\section{Elementary cellular automata} \label{sec-elementary}

Throughout this section, let $A = \{ 0,1 \}$. An elementary cellular automaton is an element $\tau \in \CA(\mathbb{Z}, A)$ with memory set $S=\{ -1, 0 ,1 \}$. These are labeled as `Rule $M$', where $M$ is a number from $0$ to $255$. In each case, the local rule $\mu_M : A^S \to A$ is determined as follows: let $M_1 \dots M_8$ be the binary representation of $M$ and write the elements of $A^S$ in lexicographical descending order, i.e. $111, 110, \dots, 000$; then, the image of the $i$-th element of $A^S$ under $\mu_M$ is $M_i$. 

In Example \ref{ex:rules128-254}, we showed that Rules 128 and 254 are both vN-regular, while, in Example \ref{ex:rule110}, we saw that Rule 110 is not vN-regular. A natural goal is to classify which ones of the $256$ elementary cellular automata are vN-regular. In other words, we ask the following question:

\begin{question}\label{question}
For which elementary cellular automata $\tau \in \CA(\mathbb{Z};A)$ there exists $\sigma  \in \CA(\mathbb{Z};A)$ (not necessarily elementary) such that $\tau \sigma \tau = \tau$?
\end{question}    

In order to achieve this goal, it is convenient to define some equivalence relations between elementary CA that are preserved by vN-regularity. In general, $\tau_1, \tau_2 \in \CA(G;A)$ are said to be \emph{conjugate} if there exists $\phi \in \ICA(G;A)$ such that $\tau_2 = \phi^{-1} \tau_1 \phi$. It is clear that in this situation, $\tau_1$ is vN-regular if and only if $\tau_2$ is vN-regular: indeed, if $\sigma \in \CA(G;A)$ is such that $\tau_1 \sigma \tau_1 = \tau_1$, then
\[ \tau_2 = \phi^{-1} \tau_1 \phi = \left( \phi^{-1} \tau_1 \phi \right) \left(\phi^{-1} \sigma \phi \right) \left(\phi^{-1} \tau_1 \phi \right) = \tau_2 \sigma^\prime \tau_2,     \]
where $\sigma^\prime = \phi^{-1} \sigma \phi \in \CA(G;A)$. 

There is another symmetry among elementary CA known as the \emph{mirrored rule}. In order to describe this, for any $x \in A^\mathbb{Z}$, denote by $x^{\rev}$ the reflection of $x$ through $0$; in other words, $(k)x^\rev = (-k) x$, for all $k \in \mathbb{Z}$. Now, for any $\tau \in \CA(\mathbb{Z};A)$, define $\tau^\star : A^\mathbb{Z} \to A^\mathbb{Z}$ by
\[ (x) \tau^\star = ((x^\rev)\tau)^\rev, \ \forall x \in A^\mathbb{Z}. \]

\begin{proposition}
The map $\tau^\star: A^\mathbb{Z} \to A^\mathbb{Z}$ is a cellular automaton. Moreover, if $\tau \in \CA(\mathbb{Z};A)$ is elementary, then $\tau^\star$ is elementary.  
\end{proposition}
\begin{proof}
The function $\rev :  A^\mathbb{Z} \to  A^\mathbb{Z}$ is continuous as its composition with the projection to $k \in \mathbb{Z}$ is equal to the projection to $-k \in \mathbb{Z}$ (which is continuous in the prodiscrete topology of $A^{\mathbb{Z}}$). Hence $\tau^\star$, being the composition of continuous functions, is also continuous. Observe that, for all $k, r \in \mathbb{Z}$, $x \in A^\mathbb{Z}$, we have
\[ (k) (x \cdot r)^\rev = (-k) (x \cdot r) = (-k-r) x = (k + r) x^\rev = (k)x^\rev \cdot (-r).   \]
Thus,
\[ (x \cdot r ) \tau^\star = ((x\cdot r)^\rev)\tau)^\rev = ((x^\rev)\tau)^\rev \cdot (-(-r)) = (x)\tau^\star \cdot r.     \]    
This proves that $\tau^\star$ is $\mathbb{Z}$-equivariant, and, by Curtis-Hedlund theorem, $\tau^\star \in  \CA(\mathbb{Z};A)$.

If $\tau$ is elementary, it is easy to see that $\tau^\star$ has memory set contained in $\{ -1, 0 ,1\}$, so it is also elementary.
\end{proof}

\begin{proposition}
Let $\tau \in \CA(\mathbb{Z};A)$. Then:
\begin{enumerate}
\item $(\tau^\star)^\star = \tau$.
\item $(\tau \circ \sigma)^\star = \tau^\star \circ \sigma^\star$, for any $\sigma \in \CA(\mathbb{Z};A)$.
\item $\tau$ is vN-regular if and only if $\tau^\star$ is vN-regular.
\end{enumerate}
\end{proposition}
\begin{proof}
It is clear that, for all $x \in A^{\mathbb{Z}}$,
\[ (x)(\tau^\star)^\star = ((x^\rev) \tau^\star)^\rev = ( (x^{\rev \ \rev}) \tau )^{\rev \ \rev} = (x) \tau.  \]
Now,
\[   (x) \tau^\star \circ \sigma^\star = (((x^\rev)\tau)^\rev) \sigma^\star = (((x^\rev)\tau)^{\rev \ \rev}) \sigma)^\rev = ((x^\rev)\tau \circ \sigma)^\rev = (x)(\tau \circ \sigma)^\star.    \]
The last part follows because $\tau \sigma \tau = \tau$ if and only if $\tau^\star \sigma^\star \tau^\star = \tau^\star$.
\end{proof}

Let $\phi_{51} \in \CA(\mathbb{Z};A)$ be Rule 51. This CA is invertible with $(\phi_{51})^{-1} = \phi_{51}$, its minimal memory set is $\{  0\}$, and can be thought as the transformation of $A^\mathbb{Z}$ that interchanges $0$'s and $1$'s. 

\begin{proposition}
Let $\tau \in \CA(\mathbb{Z};A)$. The following are equivalent:
\begin{enumerate}
\item $\tau$ is elementary and vN-regular.
\item $\phi_{51}  \tau$ is elementary and vN-regular.
\item $\tau \phi_{51}$ is elementary and vN-regular.
\end{enumerate}
\end{proposition} 
\begin{proof}
By \cite[Proposition 1.4.8]{CSC10}, a memory set for $\phi_{51} \tau$ is $\{0 \} + S = S$, where $S$ is a memory set for $\tau$. Hence, $\tau$ is elementary if and only if  $\phi_{51} \tau$ is elementary, and the other equivalence is shown similarly. 

Suppose $\tau$ is vN-regular and let $\sigma \in \CA(\mathbb{Z};A)$ be such that $\tau \sigma \tau = \tau$. Then
\[ (\phi_{51}  \tau) (\sigma \phi_{51}) (\phi_{51} \tau ) = \phi_{51} \tau \sigma \tau =\phi_{51} \tau,  \]
so $ \phi_{51} \tau$ is vN-regular. Conversely, if $ \phi_{51} \tau$ is vN-regular, let $\sigma^\prime \in \CA(\mathbb{Z};A)$ be such that $(\phi_{51} \tau) \sigma^\prime (\phi_{51} \tau) = \phi_{51} \tau$. Canceling $\phi_{51}$, we obtain that $\tau (\sigma^\prime \phi_{51})  \tau = \tau$, so $\tau$ is vN-regular. The other equivalence is shown similarly.
\end{proof}

In the literature, two elementary CA $\tau_1$ and $\tau_2$ are said to be equivalent if $\tau_1 = \tau_2$, or $\tau_1 = (\tau_2)^\star$, or $\tau_1 = \phi_{51}\tau_2$, or $\tau_1 = \phi_{51} (\tau_2)^\star$. This defines 88 equivalence classes of elementary cellular automata (see \cite[Table 1]{M13}). Here, we extend this notion of equivalence.

\begin{definition}
Let $\langle \phi_{51} \rangle = \{ \id ,  \phi_{51} \}$. Two elementary CA $\tau_1$ and $\tau_2$ are said to be \emph{equivalent} if $\tau_1 \in \langle \phi_{51} \rangle  \tau_2 \langle \phi_{51} \rangle$, or $\tau_1 \in  \langle \phi_{51} \rangle  (\tau_2)^\star \langle \phi_{51} \rangle$.
\end{definition}  

\begin{table}[!htb]
\setlength{\tabcolsep}{4pt}
\renewcommand{\arraystretch}{1.2}
\centering
\begin{tabular}{|c|c|c|c|c|c|}
\hline \textbf{ Rep. } & \textbf{ Equivalent rules } & \textbf{ Reg.} & \textbf{ Rep. } & \textbf{ Equivalent rules } & \textbf{ Reg.}\\ \hline
0  & 255  & R & 29 & 71, 184, 226  & R   \\ \hline
1 & 127, 128, 254 & R & 30 & 86, 106, 120, 135, 149, 169, 225  & NR \\ \hline
2 & 8, 16, 64,  191, 239, 247, 253   & R & 33 & 123, 132, 222  & - \\ \hline
3 & 17, 63, 119, 136, 192, 238, 252 & R & 35 & 49, 59, 115, 140, 196, 206, 220 & R  \\ \hline
4 & 32, 223, 251 & R & 36 & 219  & NR \\ \hline
5 & 95, 160, 250 & R  & 37  & 91, 164, 218  & NR   \\ \hline
6 & 20, 40, 96, 159, 215, 235, 249 & - & 38 & 44, 52, 100, 155, 203, 211, 217  & NR \\ \hline
7 & 21, 31, 87, 168, 224, 234, 248 & - & 41 & 97, 107, 121, 134, 148, 158, 214     & - \\ \hline
9 & 65, 111, 125, 130, 144, 190, 246 & - & 43 & 113, 142, 212 & R \\ \hline
10 & 80, 175, 245 & R & 45 & 75, 89, 101, 154, 166, 180, 210 & NR \\ \hline
11 & 47, 81, 117, 138, 174, 208, 244 & R & 46 & 116, 139, 209  & NR \\ \hline
12 & 34, 48, 68, 187, 207, 221, 243 & R & 50 & 76, 179, 205  & R \\ \hline
13 & 69, 79, 93, 162, 176, 186, 242 & R & 51 & 204  & R \\ \hline
14 & 42, 84, 112, 143, 171, 213, 241  & R  & 54 & 108, 147, 201  & NR  \\ \hline 
15 & 85, 170, 240 & R & 57 & 99, 156, 198  & -  \\ \hline
18 & 72, 183, 237  & NR & 58 & 78, 92, 114, 141, 163, 177, 197   & - \\ \hline
19 & 55, 200, 236 & R & 60 & 102, 153, 195 & NR \\ \hline
22 & 104, 151, 233  & NR & 62 & 110, 118, 124, 131, 137, 145, 193  & NR \\ \hline
23 & 232  & - & 73 & 109, 146, 182  & NR  \\ \hline
24 & 66, 189, 231 & NR & 77  & 178  & - \\ \hline
25 & 61, 67, 103, 152, 188, 194, 230  & NR & 90 & 165  & NR \\ \hline
26 & 74, 82, 88,  167, 173, 181, 229  & NR & 94 & 122, 133, 161  & NR \\ \hline
27 & 39, 53, 83, 172, 202, 216, 228  & - & 105 & 150 & NR \\ \hline
28 & 56, 70, 98, 157, 185, 199, 227  & - & 126 & 129 & NR \\ \hline
\end{tabular}
\caption{vN-regularity of elementary CA}
\label{elementary}
\end{table}

Thus, according to our definition, the equivalence class of an elementary CA $\tau$ is 
\[ [\tau] = \{ \tau, \tau \phi_{51}, \phi_{51} \tau, \phi_{51} \tau \phi_{51}, \tau^\star, \tau^\star \phi_{51}, \phi_{51} \tau^\star, \phi_{51} \tau^\star \phi_{51} \}. \]
Note that some of the elements in the above set might be equal. We could, potentially, try to increase this equivalence class by conjugating by another invertible CA. Let $\xi$ be the shift map of $A^\mathbb{Z}$. It is known (see \cite[Sec. 4.16]{M13}) that the invertible elementary CA are rules 15, 51, 85, 170, 204, and 240, which correspond to $\xi \phi_{51}$, $\phi_{51}$, $\xi^{-1} \phi_{51}$, $\xi^{-1}$, $\id$, and $\xi$, respectively. As $\xi$ commutes with every CA, conjugation by one of the previous does not add anything new to $[\tau]$, e.g. $(\xi \phi_{51}) \tau (\phi_{51} \xi ^{-1} )=   \phi_{51} \tau \phi_{51}$. As $(\phi_{51})^\star = \phi_{51}$, the equivalence class $[\tau]$ cannot be increased by further applying the mirrored rule.  

Our notion of equivalence defines 48 equivalence classes of elementary CA (see Table \ref{elementary}) with the property that if an element of the class is vN-regular, then all other elements of the class are vN-regular. 

\begin{lemma}
The equivalence classes of the following rules are not vN-regular: $18$, $22$, $24$, $25$, $26$, $30$, $36$, $37$, $38$, $45$, $46$, $54$, $60$, $62$, $73$, $90$, $105$, $122$ and $126$.
\end{lemma}
\begin{proof}
This result follows by the repeated use of Theorem \ref{test-non-regular}. We define configurations $y_i \in \{0,1 \}^{\mathbb{Z}}$ as follows:
\begin{align*}
(i) y_1 &= 1, \ \forall i \in \mathbb{Z} \\
 (i) y_2 &= \begin{cases}
  1 & \text{ if $i$ is even } \\
  0 & \text{ otherwise }
  \end{cases} \\
  (i)y_3 &= \begin{cases}
  1 & \text{ if $i$ is multiple of $3$ } \\
  0 & \text{ otherwise. }
  \end{cases} 
 \end{align*}
For each rule, Table \ref{table-non-regular} gives which configuration $y_i$ satisfies the hypothesis of Theorem \ref{test-non-regular}. 

\begin{table}[!htb]
\setlength{\tabcolsep}{4pt}
\renewcommand{\arraystretch}{1.2}
\centering
\begin{tabular}{|c|c||c|c||c|c||c|c|}
\hline \text{ Rule } & $x$ &  \text{ Rule } & $x$ &  \text{ Rule } & $x$ &  \text{ Rule } & $x$   \\ \hline 
 18 & $y_2$ & 26  & $y_2$  & 45 & $y_2$ & 73 & $y_2$ \\ \hline 
22 & $y_1$ & 30  & $y_1$ & 46 & $y_1$  & 90 & $y_2$ \\ \hline 
24 & $y_2$  & 36  & $y_1$ & 54 & $y_1$  & 105 & $y_3$ \\ \hline 
25 & $y_2$  & 37  & $y_2$ & 60 & $y_1$  & 122 & $y_1$  \\ \hline 
& & 38  & $y_1$  & 62 & $y_1$  & 126 & $y_1$  \\ \hline 
\end{tabular}
\caption{Elementary CA that are not vN-regular}
\label{table-non-regular}
\end{table}
\end{proof}

\begin{lemma}
The equivalence classes of the following rules are vN-regular: $0$, $2$, $4$, $5$, $10$, $11$, $12$, $13$, $14$, $15$, $29$, $35$, $43$, $51$, $76$, $128$, $192$ and $200$. 
\end{lemma}
\begin{proof}
Rules $15$ and $51$ are vN-regular because they are invertible. Rules $0$, $4$, $12$, $76$ and $200$ are idempotent (i.e. $\tau^2 = \tau$, see \cite[Table 1]{E15}). By \cite[Table 1]{E15}, rules $5$ and $29$ satisfy that $\tau^3 = \tau$, so they both are vN-regular. Rule $128$ is vN-regular by Example \ref{ex:rules128-254}. Rule $192$ has local rule $(x)\mu = \min \{ (-1)x, (0)x \}$; similar calculations as in Example \ref{ex:rules128-254} show that Rule $238$, with  local rule $(x)\mu = \max \{ (0)x , (1)x \}$, is a generalized inverse of Rule $192$.

Computer calculations obtained generalised inverses for rules $2$, $10$, $11$, $13$, $14$, $35$ and $43$. These are given in Table \ref{table-regular}. 

\begin{table}[!htb]
\setlength{\tabcolsep}{4pt}
\renewcommand{\arraystretch}{1.2}
\centering
\begin{tabular}{|c|c||c|c||c|c|}
\hline \ Rule \  &  \ Gen. Inv. \ &  \text{ Rule } &  \ Gen. Inv. \  &  \text{ Rule } & \ Gen. Inv. \   \\ \hline 
 0 & 0   & 12  & 12  & 43  & 113  \\ \hline 
 2 & 16    & 13   & 21 & 51   & 51   \\ \hline 
4 & 4    & 14   & 85  & 76  & 76    \\ \hline 
5 & 5   & 15   & 85  & 128 & 254     \\ \hline 
10 & 80  & 29   & 29   & 192 & 238  \\ \hline 
11  & 85  & 35  & 49   & 200 & 200     \\ \hline 
\end{tabular}
\caption{Elementary CA that are vN-regular}
\label{table-regular}
\end{table}

\end{proof}

This leaves $11$ classes for which we could not decide whether they were vNregular or not. Computer calculations show that no member of those classes has an elementary weak inverse.

After this paper was submitted for publication, Ville Salo in \cite{S18} decided the vN-regularity of the $11$ open classes of Table \ref{elementary}, thus completing the answer of Question \ref{question}.  


\section{Finite cellular automata} \label{finite}

In this section we characterise the vN-regular elements in the monoid $\CA(G;A)$ when $G$ and $A$ are both finite (Theorem \ref{characterisation}). In order to achieve this, we summarise some of the notation and results obtained in \cite{CRG16a,CRG16b,CRG17b}.

In the case when $G$ and $A$ are both finite, every subset of $A^G$ is closed in the prodiscrete topology, so the subshifts of $A^G$ are simply unions of $G$-orbits. Moreover, as every map $\tau : A^G \to A^G$ is continuous in this case, $\CA(G;A)$ consists of all the $G$-equivariant maps of $A^G$. Theorem \ref{conjugate} is easily deduced from Lemmas 3, 9 and 10 in \cite{CRG17b}.

\begin{theorem}\label{conjugate}
Let $G$ be a finite group of size $n \geq 2$ and $A$ a finite set of size $q \geq 2$. Let $x,y \in A^G$.
\begin{description}
\item[(i)] Let $\tau \in \CA(G;A)$. If $(x)\tau \in xG$, then $\tau \vert_{xG} \in \Sym(xG)$.
\item[(ii)] There exists $\tau \in \ICA(G;A)$ such that $(x)\tau = y$ if and only if $G_x = G_y$. 
\item[(iii)] There exists $\tau \in \CA(G;A)$ such that $(x) \tau = y$ if and only if $G_x \leq G_y$.
 \end{description} 
\end{theorem}

The following result shows that the converse of Theorem \ref{test-non-regular} holds for finite CA.

\begin{theorem}\label{characterisation}
Let $G$ be a finite group and $A$ a finite set of size $q\geq 2$. Let $\tau \in \CA(G;A)$. Then, $\tau$ is vN-regular if and only if for every $y \in (A^G)\tau$ there is $x \in A^G$ such that $(x)\tau = y$ and $G_x = G_y$.
\end{theorem}
\begin{proof}
The ``only if'' part follows by the contrapositive of Theorem \ref{test-non-regular}.  

Conversely, suppose that for every $y \in (A^G)\tau$ there is $x \in A^G$ such that $(x)\tau = y$ and $G_x = G_y$. Choose pairwise distinct $G$-orbits $y_1G, \dots, y_\ell G$ such that  
\[ (A^G)\tau = \bigcup_{i=1}^{\ell} y_i G. \]
For each $i$, fix $y_i^\prime \in A^G$ such that $(y_i^\prime) \tau =  y_i$ and $G_{y_i}= G_{y_i^\prime}$. We define $\phi : A^G \to A^G$ as follows: for any $z \in A^G$,
\[ (z)\phi := \begin{cases}
z & \text{if } z \not \in (A^G)\tau, \\
y_i^\prime \cdot g & \text{if } z = y_i \cdot g \in y_i G.
 \end{cases} \]
The map $\phi$ is well-defined because
\[ y_i \cdot g = y_i \cdot h \ \Longleftrightarrow \ hg^{-1} \in G_{y_i} = G_{y_i^\prime} \ \Longleftrightarrow  \ y_i^\prime \cdot g = y_i^\prime \cdot h.  \]
Clearly, $\phi$ is $G$-equivariant, so $\phi \in \CA(G;A)$. Now, for any $x \in A^G$ with $(x)\tau = y_i \cdot g$, 
\[ (x) \tau \phi \tau = (y_i \cdot g) \phi \tau = (y_i^\prime \cdot g) \tau = (y_i^\prime) \tau \cdot g =  y_i \cdot g = (x) \tau.  \]
This proves that $\tau \phi \tau = \tau$, so $\tau$ is vN-regular. 
\end{proof}

Our goal now is to find a vN-regular submonoid of $\CA(G;A)$ and describe its structure (see Theorem \ref{le:max vN-regular}). In order to achieve this, we need some further terminology and basic results. 

Say that two subgroups $H_1$ and $H_2$ of $G$ are \emph{conjugate} in $G$ if there exists $g \in G$ such that $g^{-1} H_1 g = H_2$. This defines an equivalence relation on the subgroups of $G$. Denote by $[H]$ the conjugacy class of the subgroup $H$ of $G$.
Define the \emph{box} in $A^G$ corresponding to $[H]$, where $H \leq G$, by 
\[ B_{[H]}(G;A) := \{ x \in A^G : [G_x] =  [H] \}. \]
As any subgroup of $G$ is the stabiliser of some configuration in $A^G$, the set $\{ B_{[H]}(G;A) : H \leq G \}$ is a partition of $A^G$. Note that $B_{[H]}(G;A)$ is a subshift of $A^G$ (because $G_{(x \cdot g)} = g^{-1} G_x g$) and, by the Orbit-Stabiliser Theorem, all the $G$-orbits contained in $B_{[H]}(G;A)$ have equal sizes. When $G$ and $A$ are clear from the context, we write simply $B_{[H]}$ instead of $B_{[H]}(G;A)$. 

\begin{example}
For any finite group $G$ and finite set $A$ of size $q$, we have 
\[ B_{[G]} = \{ \mathbf{k} \in A^G : \mathbf{k} \text{ is constant} \}. \]
\end{example}

Let $\alpha_{[H]}(G;A)$ be the number of $G$-orbits inside the box $B_{[H]}$; for example, $\alpha_{[G]}(G; A) = \vert A \vert$, as every constant configuration defines an orbit of size $1$.

A submonoid $R \leq M$ is called \emph{maximal vN-regular} if there is no vN-regular monoid $K$ such that $R < K < M$.  

\begin{theorem} \label{le:max vN-regular}
Let $G$ be a finite group and $A$ a finite set of size $q\geq 2$. Let
\[ R := \left\{ \sigma \in \CA(G;A) : G_x = G_{(x)\sigma} \text{ for all } x \in A^G  \right\}. \]
\begin{description}
\item[(i)] $R = \left\{  \sigma \in \CA(G;A)  : (B_{[H]})\sigma \subseteq B_{[H]}, \ \forall H \leq G  \right\}$. 
\item[(ii)] $\ICA(G;A) \leq R$.
\item[(iii)] $R$ is a vN-regular monoid. 
\item[(iv)] $R$ is not a maximal vN-regular submonoid of $\CA(G;A)$.
\end{description}
\end{theorem}
\begin{proof}
Parts \textbf{(i)} and \textbf{(ii)} follow by Theorem \ref{conjugate}, while part \textbf{(iii)} follows by Theorem \ref{characterisation}.

For part \textbf{(iv)}, let $x, y \in A^G$ be such that $G_x < G_y$, so $x$ and $y$ are in different boxes. Define $\tau \in \CA(G;A)$ such that $(x) \tau = y$, $(B_{[G_y]})\tau = yG$, and $\tau$ fixes any other configuration in $A^G \setminus (B_{[G_y]} \cup \{ xG\})$. It is clear by Theorem \ref{characterisation} that $\tau$ is vN-regular. We will show that $K:=\langle R, \tau \rangle$ is a vN-regular submonoid of $\CA(G;A)$. Let $\sigma \in K$ and $z \in (A^G)\sigma$. If $\sigma \in R$, then it is obviously vN-regular, so assume that $\sigma = \rho_1 \tau \rho_2$ with $\rho_1 \in K$ and $\rho_2 \in R$. If $z \in A^G \setminus (B_{[G_y]})$, it is clear that $z$ has a preimage in its own box; otherwise $(B_{[G_y]}) \sigma = (yG) \rho_2 = zG$ and $z$ has a preimage in $B_{[G_y]}$. Hence $\sigma$ is vN-regular and so is $K$. 
\end{proof}

For any integer $\alpha \geq 2$ and any group $C$, the \emph{wreath product} of $C$ by $\Tran_\alpha$ is the monoid
\[ C \wr \Tran_{\alpha} := \{ (v, f) : v \in C ^\alpha, f \in \Tran_\alpha \} \]
equipped with the operation
\[ (v, f) \cdot (w, g) = ( v (f \cdot w),  f \circ g ), \text{ for any } v,w \in C^\alpha, f,g \in \Tran_\alpha, \]
where $f$ acts on the left on $w$ as follows:
\[  f \cdot w = f \cdot (w_1, w_2, \dots, w_\alpha) := (w_{(1)f}, w_{(2)f}, \dots, w_{(\alpha)f}). \]
For a more detailed description of the wreath product of monoids see \cite[Sec. 2]{AS09}.

For any $H \leq G$, let $N_G(H):= \{ g \in G : g H g^{-1} = H \}$ be the normaliser of $H$ in $G$. 

\begin{theorem}\label{structure-monoid}
Let $G$ be a finite group, $A$ a finite set of size $q\geq 2$ and $R \leq \CA(G;A)$ as given by Theorem \ref{le:max vN-regular}. Let $H_1, H_2, \dots, H_r$ be a complete list of representatives for the conjugacy classes of subgroups of $G$. Then:
\[ R \cong \prod_{i=1}^r (N_G(H_i)/ H_i) \wr \Tran_{\alpha_i}, \]
where $\alpha_i = \alpha_{[H_i]}(G;A)$.
\end{theorem}
\begin{proof}
By Theorem \ref{le:max vN-regular} \textbf{(i)}, we see that $R$ may be embedded in $\prod_{i=1}^n \Tran(B_{[H_i]})$. Any CA preserves the uniform partition of $B_{[H_i]}$ into $G$-orbits, so by \cite[Lemma 2.1]{AS09}, the projection of $R$ to $\Tran(B_{[H_i]})$ is contained in $\Tran(O) \wr \Tran_{\alpha}$, where $O$ is a fixed orbit contained in $B_{[H]}$. Any CA acts on the orbit $O$ by $N_G(H_i)/ H_i$ and it may induce any transformation among orbits (c.f. the proof of Theorem 3 in \cite{CRG17b}). Thus, the projection of $R$ to $\Tran(B_{[H_i]})$ is precisely $ (N_G(H_i)/ H_i) \wr \Tran_{\alpha_i}$.
\end{proof}


\section{Linear cellular automata} \label{linear}

Let $V$ a vector space over a field $\mathbb{F}$. For any group $G$, the configuration space $V^G$ is also a vector space over $\mathbb{F}$ equipped with the pointwise addition and scalar multiplication. Denote by $\End_{\mathbb{F}}(V^G)$ the set of all $\mathbb{F}$-linear transformations of the form $\tau : V^G \to V^G$. Define
\[ \LCA(G;V) := \CA(G;V) \cap \End_{\mathbb{F}}(V^G). \]
Note that $\LCA(G;V)$ is not only a monoid, but also an $\mathbb{F}$-algebra (i.e. a vector space over $\mathbb{F}$ equipped with a bilinear binary product), because, again, we may equip $\LCA(G;V)$ with the pointwise addition and scalar multiplication. In particular, $\LCA(G;V)$ is also a ring. 

As in the case of semigroups, vN-regular rings have been widely studied and many important results have been obtained. In this chapter, we study the vN-regular elements of $\LCA(G;V)$ under some natural assumptions on the group $G$. 

First, we introduce some preliminary results and notation. Let $R$ be a ring. The \emph{group ring} $R[G]$ is the set of all functions $f : G \to R$ with finite support (i.e. the set $\{ g \in G : (g)f \neq 0 \}$ is finite). Equivalently, the group ring $R[G]$ may be defined as the set of all formal finite sums $\sum_{g \in G} a_g g$ with $a_g \in R$. The multiplication in $R[G]$ is defined naturally using the multiplications of $G$ and $R$:
\[ \sum_{g \in G} a_g g \sum_{h \in G} a_h h = \sum_{g, h \in G} a_g a_h  gh. \]
 If we let $R :=\End_{\mathbb{F}}(V)$, it turns out that $\End_{\mathbb{F}}(V)[G]$ and $\LCA(G;V)$ are isomorphic as $\mathbb{F}$-algebras (see \cite[Theorem 8.5.2]{CSC10}). 

Define the \emph{order} of $g \in G$ by $o(g) :=  \vert \langle g \rangle \vert$ (i.e. the size of the subgroup generated by $g$). The group $G$ is \emph{torsion-free} if the identity is the only element of finite order; for instance, the groups $\mathbb{Z}^d$, for $d \in \mathbb{N}$, are torsion-free groups. The group $G$ is \emph{elementary amenable} if it may be obtained from finite groups or abelian groups by a sequence of group extensions or direct unions.   

In the following theorem we characterise the vN-regular linear cellular automata over fields and torsion-free elementary amenable groups (such as $\mathbb{Z}^d$, $d \in \mathbb{N}$).

\begin{theorem}\label{torsion-free}
Let $G$ be a torsion-free elementary amenable group and let $V = \mathbb{F}$ be any field. A non-zero element $\tau \in \LCA(G; \mathbb{F})$ is vN-regular if and only if it is invertible.
\end{theorem}
\begin{proof}
It is clear that any invertible element is vN-regular. Let $\tau \in \LCA(G;\mathbb{F})$ be non-zero vN-regular. In this case, $\End_{\mathbb{F}}(\mathbb{F}) \cong \mathbb{F}$, so $\LCA(G;\mathbb{F}) \cong \mathbb{F}[G]$. By definition, there exists $\sigma \in \LCA(G;V)$ such that $\tau \sigma \tau = \tau$. As $\LCA(G;V)$ is a ring, the previous implies that
\[ \tau (\sigma\tau - 1) = 0 \ \text{ and } \ (\tau \sigma - 1) \tau = 0, \]
where $1 = 1e$ and $0=0e$ are the identity and zero endomorphisms, respectively. It was established in \cite[Theorem 1.4]{KLM88} that $\mathbb{F}[G]$ has no zero-divisors whenever $G$ is a torsion-free elementary amenable group. As $\tau \neq 0$, then $\sigma\tau - 1 = 0$ and $\tau \sigma - 1 = 0$, which means that $\tau$ is invertible.
\end{proof}

The argument of the previous result works as long as the group ring $\mathbb{F}[G]$ has no zero-divisor. This is connected with the well-known Kaplansky's conjecture which states that $\mathbb{F}[G]$ has no zero-divisors when $G$ is a torsion-free group.    

The \emph{characteristic} of a field $\mathbb{F}$, denoted by $\Char(\mathbb{F})$, is the smallest $k \in \mathbb{N}$ such that 
\[ \underbrace{1+1+\cdots+1}_{k \text{ times}} = 0 ,\]
where $1$ is the multiplicative identity of $\mathbb{F}$. If no such $k$ exists we say that $\mathbb{F}$ has characteristic $0$.

A group $G$ is \emph{locally finite} if every finitely generated subgroup of $G$ is finite; in particular, the order of every element of $G$ is finite. Examples of such groups are finite groups and infinite direct sums of finite groups.

\begin{theorem}\label{vN-regular-ring}
Let $G$ be a group and let $V$ be a finite-dimensional vector space over $\mathbb{F}$. Then, $\LCA(G;V)$ is vN-regular if and only if $G$ is locally finite and $\Char(\mathbb{F}) \nmid o(g)$, for all $g \in G$.
\end{theorem}
\begin{proof}
By \cite[Theorem 3]{C63} (see also \cite{A,ML}), we have that a group ring $R[G]$ is vN-regular if and only if $R$ is vN-regular, $G$ is locally finite and $o(g)$ is a unit in $R$ for all $g \in G$. In the case of $\LCA(G;V) \cong \End_{\mathbb{F}}(V)[G]$, since $\dim(V) := n < \infty$, the ring $R:=\End_{\mathbb{F}}(V) \cong M_{n\times n}(\mathbb{F})$ is vN-regular (see \cite[Theorem 1.7]{G79}. The condition that $o(g)$, seen as the matrix $o(g) I_n$, is a unit in $M_{n\times n}(\mathbb{F})$ is satisfied if and only if $o(g)$ is nonzero in $\mathbb{\mathbb{F}}$, which is equivalent to $\Char(\mathbb{F}) \nmid o(g)$, for all $g \in G$. 
\end{proof}

\begin{corollary}\label{cor-locally-finite}
Let $G$ be a group and let $V$ be a finite-dimensional vector space over a field $\mathbb{F}$ of characteristic $0$. Then, $\LCA(G;V)$ is vN-regular if and only if $G$ is locally finite. 
\end{corollary}

Henceforth, we focus on the vN-regular elements of $\LCA(G;V)$ when $V$ is a one-dimensional vector space (i.e. $V$ is just the field $\mathbb{F}$). In this case, $\End_{\mathbb{F}}(\mathbb{F}) \cong \mathbb{F}$, so $\LCA(G;\mathbb{F})$ and $\mathbb{F}[G]$ are isomorphic as $\mathbb{F}$-algebras. 

A non-zero element $a$ of a ring $R$ is called \emph{nilpotent} if there exists $n > 0$ such that $a^n = 0$. The following basic result will be quite useful in the rest of this section.

\begin{lemma}\label{nilpotent}
Let $R$ be a commutative ring. If $a \in R$ is nilpotent, then $a$ is not a vN-regular element.
\end{lemma}
\begin{proof}
Let $R$ be a commutative ring and $a \in R$ a nilpotent element. Let $n >0$ be the smallest integer such that $a^n = 0$. Suppose $a$ is a vN-regular element, so there is $x \in R$ such that $axa=a$. By commutativity, we have $a^2 x = a$. Multiplying both sides of this equation by $a^{n-2}$ we obtain $0 = a^{n}  x = a^{n-1}$, which contradicts the minimality of $n$. 
\end{proof}

\begin{example}
Suppose that $G$ is a finite abelian group and let $\mathbb{F}$ be a field such that $\Char(\mathbb{F}) \mid \vert G \vert$. By Theorem \ref{vN-regular-ring}, $\LCA(G;\mathbb{F})$ must have elements that are not vN-regular. For example, let $s := \sum_{g \in G} g \in \mathbb{F}[G]$. As $sg = s$, for all $g \in G$, and $\Char(\mathbb{F}) \mid \vert G \vert$, we have $s^2 = \vert G \vert s = 0$. Clearly, $\mathbb{F}[G]$ is commutative because $G$ is abelian, so, by Lemma \ref{nilpotent}, $s$ is not a vN-regular element. 
\end{example}

We finish this section with the special case when $G$ is the cyclic group $\mathbb{Z}_n$ and $\mathbb{F}$ is a finite field with $\Char(\mathbb{F}) \mid n$. By Theorem \ref{vN-regular-ring}, not all the elements of $\LCA(\mathbb{Z}_n; \mathbb{F})$ are vN-regular, so how many of them are there? In order to count them we need a few technical results about commutative rings.

An \emph{ideal} $I$ of a commutative ring $R$ is a subring such that $rb  \in I$ for all $r \in R$, $b \in I$. For any $a \in R$, the \emph{principal ideal} generated by $a$ is the ideal $\langle a \rangle := \{ ra : r \in R \}$. A ring is called \emph{local} if it has a unique maximal ideal. 

Denote by $\mathbb{F}[x]$ the ring of polynomials with coefficients in $\mathbb{F}$. When $G \cong \mathbb{Z}_n$, we have the following isomorphisms as $\mathbb{F}$-algebras:
\[ \LCA(\mathbb{Z}_n;\mathbb{F}) \cong \mathbb{F}[\mathbb{Z}_n] \cong \mathbb{F}[x] / \langle x^n -1 \rangle, \]
where $\langle x^n -1 \rangle$ is a principal ideal in $\mathbb{F}[x]$.

\begin{theorem} 
Let $n \geq 2$ be an integer, and let $\mathbb{F}$ be a finite field of size $q$ such that $\Char(\mathbb{F}) \mid n$. Consider the following factorization of $x^n - 1$ into irreducible elements of $\mathbb{F}[x]$: 
\[ x^n -1 = p_1(x)^{m_1} p_2(x)^{m_2} \dots p_r(x)^{m_r}. \]
For each $i = 1, \dots r$, let $d_i := \deg(p_i(x))$. Then, the number of vN-regular elements in $\LCA(\mathbb{Z}_n; \mathbb{F})$ is exactly
\[ \prod_{i=1}^{r} \left( ( q^{d_i} - 1) q^{d_i (m_i-1)} +1 \right). \]
\end{theorem}
\begin{proof}
Recall that
\[ \LCA(\mathbb{Z}_n;\mathbb{F}) \cong \mathbb{F}[x] / \langle x^n -1 \rangle. \]
By the Chinese Remainder Theorem,
\[ \mathbb{F}[x] / \langle x^n - 1 \rangle \cong \mathbb{F}[x]/ \langle p_1(x)^{m_1} \rangle \times \mathbb{F}[x]/ \langle p_2(x)^{m_2} \rangle \times \dots \times \mathbb{F}[x]/ \langle p_r(x)^{m_r} \rangle. \]
An element $a = (a_1, \dots, a_r)$ in the right-hand side of the above isomorphism is a vN-regular element if and only if $a_i$ is a vN-regular element in $\mathbb{F}[x]/ \langle p_i(x)^{m_i} \rangle$ for all $i = 1, \dots,r$. 

Fix $m := m_i$, $p(x) = p_i(x)$, and $d:=d_i$. Consider the principal ideals $A := \langle p(x) \rangle$ and $B := \langle p(x)^m \rangle$ in $\mathbb{F}[x]$. Then, $\mathbb{F}[x]/B$ is a local ring with unique maximal ideal $A/B$, and each of its nonzero elements is either nilpotent or a unit (i.e. invertible): in particular, the set of units of $\mathbb{F}[x]/B$ is precisely $(\mathbb{F}[x]/B) - (A/B)$. By the Third Isomorphism Theorem, $(\mathbb{F}[x]/B)/(A/B) \cong (\mathbb{F}[x]/A)$, so 
\[ \vert A/B \vert = \frac{\vert \mathbb{F}[x]/B \vert}{\vert \mathbb{F}[x]/A\vert} = \frac{q^{dm}}{q^d} = q^{d(m-1)}. \]
Thus, the number of units in $\mathbb{F}[x]/B$ is 
\[ \vert (\mathbb{F}[x]/B) - (A/B) \vert = q^{dm} - q^{d(m-1)} = (q^d - 1)q^{d(m-1)}. \] 

As nilpotent elements are not vN-regular by Lemma \ref{nilpotent}, every vN-regular element of $\mathbb{F}[x]/ \langle p_i(x)^{m_i} \rangle$ is zero or a unit. Thus, the number of vN-regular elements in $\mathbb{F}[x]/ \langle p_i(x)^{m_i} \rangle$ is $( q^{d_i} - 1) q^{d_i (m_i-1) } +1$.  
\end{proof}

\section{Conclusions and future work}

We studied generalised inverses and von Neumann regular cellular automata over configuration spaces $A^G$. Our main results are the following:
\begin{enumerate}
\item All cellular automata over $A^G$ are vN-regular if and only if $\vert G \vert = 1$ or $\vert A \vert =1$ (Theorem \ref{th:vN-regular}).
\item Out of the $256$ elementary cellular automata over $\{ 0,1 \}^{\mathbb{Z}}$, at least $96$ are vN-regular and $92$ are not vN-regular (Table \ref{elementary}). 
\item If $G$ and $A$ are finite, a cellular automaton $\tau$ over $A^G$ is vN-regular if and only if for every $y \in (A^G)\tau$ there is $x \in A^G$ such that $(x)\tau = y$ and $G_x = G_y$ (Theorem \ref{characterisation}).
\item If $G$ is a torsion-free elementary amenable group and $A$ is a field, a non-zero linear cellular automaton $\tau$ over $A^G$ is vN-regular if and only if it is invertible (Theorem \ref{torsion-free}).
\item If $A$ is a finite-dimensional vector space over a field of characteristic $0$, all linear cellular automata over $A^G$ are vN-regular if and only if $G$ is locally finite (Corollary \ref{cor-locally-finite}).
\end{enumerate}

It could not be determined whether $68$ elementary cellular automata (grouped in $11$ equivalence classes) are vN-regular or not. A clear direction for future work is to examine in more detail these cases. In order to decide their vN-regularity, we need new tools to find generalised inverses and an exhaustive analysis of their behavior on periodic orbits (in order to be able to use Theorem \ref{test-non-regular}).  

\section{Acknowledgment}

We thank Alberto Dennunzio, Enrico Formenti, Luca Manzoni, Luca Mariot and Antonio E. Porreca for the successful organisation of the 23rd International Workshop on Cellular Automata and Discrete Complex Systems (AUTOMATA 2017), in which some of the results of this paper were presented and discussed.


\Addresses

\end{document}